\documentclass{amsart}

\usepackage{amsthm, amssymb, amsmath}
\usepackage{a4wide}
\usepackage{array}
\usepackage[dvipsnames]{xcolor}
\usepackage[english]{babel}
\usepackage{graphicx, tikz, pgfplots}

\usepackage{subfigure}
\usepackage{url}
\usepackage{fancyhdr}
\usetikzlibrary{fit,calc,math}

\newtheorem{theorem}{Theorem}[section]

\newtheorem{observation}[theorem]{Observation}
\newtheorem{proposition}[theorem]{Proposition}
\newtheorem{corollary}[theorem]{Corollary}
\newtheorem{definition}{Definition}
\newtheorem{remark}{Remark}

\newcommand{\N}{\mathbb{N}}

\begin{document}

\title{(non)-Matching and (non)-periodicity  for $(N,\alpha)$-expansions}
\author{Cor Kraaikamp and Niels Langeveld}
\address[Cor Kraaikamp]{Delft University of Technology, EWI (DIAM), Mekelweg 4, 2628 CD Delft, the Netherlands}
\email[Cor Kraaikamp]{C.Kraaikamp@tudelft.nl}
\address[Niels Langeveld]{ Montanuniversit\"at Leoben, Department Mathematik und Informationstechnologie,
Franz-Josef-Strasse 18
 A-8700 Leoben
AUSTRIA}
\email[Niels Langeveld]{niels.langeveld@unileoben.ac.at}

\date{Version of \today}


\maketitle
 
\begin{abstract}

Recently a new class of continued fraction algorithms, the $(N,\alpha$)-expansions, was introduced in~\cite{KL} for each $N\in\N$, $N\geq 2$ and $\alpha \in (0,\sqrt{N}-1]$. Each of these continued fraction algorithms has only finitely many possible digits. These $(N,\alpha)$-expansions `behave' very different from many other (classical) continued fraction algorithms; see also~\cite{N,dJKN,dJK,CK} for examples and results. In this paper we will show that when all digits in the digit set are co-prime with $N$, which occurs in specified intervals of the parameter space, something extraordinary happens. Rational numbers and certain quadratic irrationals will not have a periodic expansion. Furthermore, there are no matching intervals in these regions. This contrasts sharply with the regular continued fraction and more classical parameterised continued fraction algorithms, for which often matching is shown to hold for almost every parameter. On the other hand, for $\alpha$ small enough, all rationals have an eventually periodic expansion with period 1. This happens for all $\alpha$ when $N=2$. We also find infinitely many matching intervals for $N=2$, as well as rationals that are not contained in any matching interval.


\end{abstract}

\maketitle

\section*{Introduction}
As a variation on the regular continued fraction, $N$-continued fractions were introduced in \cite{BGRKWY} as continued fractions of the form
\begin{equation}
x=\cfrac{N}{d_1+\cfrac{N}{d_2+\cfrac{N}{\ddots}}}\label{introductionexpansion},
\end{equation}
where $N\in \mathbb{N}$, $N\geq 2$, and the digits $d_i$ (aka partial quotients) are positive integers. In contrast with regular continued fractions (the case $N=1$), for $N\geq 2$ real numbers $x\in (0,N)$ have infinitely many different expansions of the form~\eqref{introductionexpansion}. This was first established in~\cite{AW}, and can easily be shown by describing dynamical systems that generate these continued fractions; see~\cite{DKW}. 
Of particular interest of study has been the periodicity of quadratic irrationals ($x\in \mathbb{R}\setminus \mathbb{Q}$ that are a solution to an equation $ax^2+bx+c=0$ with $a,b,c\in\mathbb{Z}$). For regular continued fractions we know that a number is eventually periodic if and only if it is a quadratic irrational. Here the situation is different. In \cite{BGRKWY}  the authors prove that for every quadratic irrational number $x$ there exist infinitely many positive integers $N$ for which the NCF expansions of $x$ is eventually periodic with period-length 1. On the other hand, when using the greedy map $T(x)=\frac{N}{x}-\lfloor \frac{N}{x} \rfloor$ for generating the continued fraction, the authors of \cite{AW}  conjecture there are expansions of quadratic irrationals that are not periodic. This was further investigated in \cite{DKW} where this conjecture was further supported by means of numerical simulations. In this article, we will show that for $(N,\alpha)$-expansions, which are another family of $N$-expansions,  we can find quadratic irrationals with \emph{no}  periodic expansion. In \cite{KL} the $(N,\alpha)$-continued fractions were introduced in the following way. Let $\alpha\in(0,\sqrt{N}-1]$, and define the map $T_{N,\alpha}:[\alpha,\alpha+1]\rightarrow[\alpha,\alpha+1)$ as
\[
T_{N,\alpha}(x)=\frac{N}{x}-\left\lfloor \frac{N}{x}-\alpha\right\rfloor .
\]
Let\footnote{If $\frac{N}{\alpha} - \alpha \in \N$ we set $d_1(\alpha)= \left\lfloor \frac{N}{x}-\alpha\right\rfloor -1$ to avoid a cylinder consisting of only the point $\alpha$; cf.~\cite{dJKN}.} $d_1(x)= \left\lfloor \frac{N}{x}-\alpha\right\rfloor$ and $d_n(x)= d_1(T^{n-1}_{N,\alpha}(x))$ for $n\geq 2$. Note that that all expansions are infinite and unique as $0\not\in [\alpha,\alpha+1]$. Now 
\[
T_{N,\alpha}(x)=\frac{N}{x}-d_1(x) 
\]
gives
\[
x=\frac{N}{d_1(x)+T_{N,\alpha}(x)}
\]
and by using this equation iteratively we find 
\[
x=\frac{N}{\displaystyle d_1(x)+\frac{N}{\displaystyle d_2(x)+\frac{N}{\displaystyle d_3(x)+\ddots}}}.
\]
The digit set is given by $\mathcal{D}_{N,\alpha}:=\{\lfloor \frac{N}{\alpha+1}-\alpha\rfloor, \ldots,  \lfloor \frac{N}{\alpha}-\alpha\rfloor  \}$ which is a finite set of consecutive positive integers. Note that  $0<\alpha \leq \sqrt{N}-1$ ensures that all digits are strictly positive.

In \cite{KL} it was shown that for certain choices of $N$ and $\alpha$ the absolute continuous invariant measure associated with the dynamical system does not have full support on $[\alpha,\alpha+1]$. In these case one or more so-called \emph{gaps} in the attractor of $T_{\alpha}$ on $[\alpha,\alpha+1)$ appear; one or more intervals where the $T_{\alpha}$-invariant measure is zero. This was further investigated in~\cite{dJKN}, where it was shown that $T_{\alpha}$ is always gapless when it has at least five cylinders, i.e., when $\lfloor N/ \alpha\rfloor - \lfloor N/(\alpha +1)\rfloor \geq 4$. For four cylinders, there are certain cases with a large gap. The cases of two, three and four cylinders are studied in detail, and sufficient conditions for gaplessness are given. In~\cite{dJK} it is shown that for $\alpha=\alpha_{\max}=\sqrt{N}-1$ the number of gaps is a finite, monotonically non-decreasing and unbounded function of $N$.


Starting with \cite{CMPT}, recently of lot of attention has been given to the study of matching for parameterized continued fraction algorithms. Examples of such families of continued fraction algorithms are Nakada's continued fractions \cite{N}, and similarly defined continued fractions but giving rise to infinite dynamical systems \cite{KLMM}, Katok and Ugarcovici's continued fractions \cite{KU}, and Tanaka and Ito's continued fractions \cite{TI} to name a few. Matching, for our purposes, can be defined as follows.
\begin{definition}[Matching]\label{def:match}
We say that (stable) matching holds for $\alpha$ if there are $K,L$ such that
$T_{\alpha}^{K}(\alpha)=T_{\alpha}^L(\alpha+1)$ and there is an $\varepsilon>0$ such that for all $\alpha^\prime\in(\alpha-\varepsilon,\alpha+\varepsilon)$ we also have $T_{\alpha^\prime}^{K}(\alpha^\prime)=T_{\alpha^\prime}^L(\alpha^\prime+1)$.
 The numbers $K,L$ are called the matching exponents, $K-L$   is called the matching index and the largest interval $(c,d)$ such that for all $\alpha\in(c,d)$ we have the same matching exponents is called a matching interval.
\end{definition}
For each of the above mentioned families, when $\alpha$ and $\alpha+1$ are replaced by the endpoints of the intervals of the dynamical system under consideration, study showed that the matching intervals cover the entire parameter space except for a set of Lebesgue measure zero. In other words, matching holds almost everywhere. This was shown in~\cite{CT,KLMM,CIT,CLS} respectively. One might start to believe that for any reasonable parameterized dynamical system that gives rise to continued fractions we will find matching almost surely. In \cite{KL} a matching interval is given for $N=2$ and for parameters in this matching interval the natural extension are build. In \cite{CK}, matching intervals are found for every $N$ and natural extensions are build. However, the matching intervals from \cite{KL,CK} are such that the matching index is 0, and their total Lebesgue measure is less than $\sqrt{N}-1$. In this article we prove that for many choices of $N\geq 2$ there are regions of the parameter space $(0,\sqrt{N}-1]$ where there are no matching intervals.

In Section \ref{sec:prelim} we will introduce notation and give some basic properties that hold for all $N$ and $\alpha\in(0,\sqrt{N}-1]$. In Section \ref{sec:nodiv} we zoom in on the case where $N$ is relatively prime with any digit from the digit set $\mathcal{D}_{N,\alpha} $. As will be shown, this is the situation where you have no matching intervals and no periodicity of rational numbers as well as the situation where you can find quadratic irrationals that are not periodic. In Section \ref{sec:N2} we investigate the case of $N=2$ for which all rational numbers are eventually mapped to $1$, and we give infinitely many matching intervals. We also show that for $N=2$ there are infinitely many \emph{bad rationals}; these are rationals that are not in any
matching interval.

\section{preliminary results}\label{sec:prelim}
Fix $N\in\N$, $N\geq 2$, $\alpha\in(0,\sqrt{N}-1]$ and $x\in[\alpha,\alpha+1]$. For readability we suppress the dependence of the variables $N,\alpha$ and $x$ in the notation. 
We view matrices as M\"obius transformations so that 
\[
A(x)=\left(\begin{matrix}
a & b \\
c & d
\end{matrix}\right)(x)=\frac{ax+b}{cx+d}.
\]
Let us define the following matrices:
\[
B_d=
\left(\begin{matrix}
0 & N \\
1 & d
\end{matrix}\right)
\]
and
\begin{equation}\label{eq:Mxalphan}
M_n=M_{\alpha,x,n}=B_{d_{1}}B_{d_{2}}\cdots B_{d_{n}}.
\end{equation}
Similar to the regular continued fraction case, one can check that 
\[
M_{n}(0)=\frac{N}{\displaystyle d_1+\frac{N}{\displaystyle d_2+\frac{N}{\displaystyle d_3+\ddots \frac{N}{d_n}}}}
\]
gives the $n^{\text{th}}$ convergent $c_n$, which will be denoted by $[0;d_1,\dots,d_n]_N$. In fact, $M_{n}$ is given by
\begin{equation}\label{Mn}
M_{n} = \left(\begin{matrix}
p_{n-1} & p_n  \\
q_{n-1} & q_n
\end{matrix}\right)
\end{equation}
where the sequences $(p_n)_{n \geq 1}$ and $(q_n)_{n \geq 1}$ satisfy the recurrence relations

\begin{eqnarray}
    p_{-1}=1, & p_0=0, & p_n=d_np_{n-1}+Np_{n-2},~\text{for } n\geq 1,\label{eq:recforpn}\\
    q_{-1}=0, & q_0=1, & q_n=d_nq_{n-1}+Nq_{n-2},~\text{for } n\geq 1,\label{eq:recforqn}
\end{eqnarray}
and we have  $c_n=\frac{p_n}{q_n}$. We have that
\begin{equation}\label{eq:noncoprimeprod}
    \det(M_n)=\det(B_{d_1}\cdots B_{d_n})=p_{n-1}q_{n}-p_nq_{n-1}=(-N)^n
\end{equation}
so that, in contrast to the regular continued fraction (the case $N=1$),  $p_n$ and $q_n$ are not necessarily co-prime (they might both have $N$ or divisors of $N$ as a common divisor).  We also have that
\begin{equation}\label{eq:xasMnTn}
    x=M_n(T_{N,\alpha}^n(x)) .
\end{equation}
Using~\eqref{eq:recforpn},~\eqref{eq:recforqn},~\eqref{eq:noncoprimeprod}, and~\eqref{eq:xasMnTn} one can show now that $\lim_{n\to\infty} c_n = x$; c.f.~\cite{DKW}.

\subsection{Periodicity of rational numbers for small $\alpha$}
\begin{proposition}\label{prop:FiniteExpansion}
Let $N \geq 2$ and  $0 < \alpha \leq \xi_N-1$, where $\xi_N$ is a positive number such that $\xi_N=\frac{N}{N-2+\xi_N}$ 
and let $t_0,s_0\in\N$ be such, that ${\rm gcd}\{ t_0,s_0\}=1$ and $t_0/s_0\in [\alpha ,\alpha +1]$. Set for $k\geq 1$, 
$$
\frac{t_k}{s_k} = T_{N,\alpha}^k\left( \tfrac{t_0}{s_0}\right),
$$
where $t_k,s_k\in\N$ and ${\rm gcd}\{ t_k,s_k\} = 1$. Then there exists an $n\in\N$, such that $t_n=s_n=1$, and thus that the $(N,\alpha$)-expansion of $t_0/s_0$ is (eventually) periodic, with period-length 1, and where the period consists of only the digit $N-1$:
$$
\frac{t_0}{s_0} = [0;d_1,\dots,d_n,\overline{N-1}]_N;
$$
(as usual the bar indicates the period).
\end{proposition}

\begin{proof}
Note that $\xi_N$ is a quadratic irrational (as $N^2+4\neq \Box$) and that $1<\xi_N=\frac{-(N-2) + \sqrt{N^2+4}}{2}<2$. So if $0<\alpha\leq \xi_N-1$ we have that $0<\alpha <1$. In case $t_0/s_0=1$ and ${\rm gcd}\{ t_0,s_0\}=1$ (and thus that $t_0=1=s_0$), we see that $T_{N,\alpha}(1) = \tfrac{N}{1}-(N-1) = 1$, and therefore $1 = [0;\overline{N-1}]_N$. We will consider the cases $\tfrac{t_0}{s_0}\in [\alpha,1)$ and $\tfrac{t_0}{s_0}\in (1,\alpha +1)$ separately.\smallskip\

In case $\tfrac{t_0}{s_0}\in [\alpha,1)$ we find that
$$
\tfrac{t_1}{s_1} = T_{N,\alpha}(\tfrac{t_0}{s_0}) = \frac{Ns_0-d_1t_0}{t_0},
$$
and since ${\rm gcd}\{t_1,s_1\} =1$, $s_1|t_0$, and $t_0/s_0<1$, we find that
$$
s_1\leq t_0<s_0.
$$
By definition of $\xi_N$ we see that one of the $N$-expansions of $\xi_N$ is given by:
\[
\xi_N = \frac{N}{\displaystyle N-2+\frac{N}{\displaystyle N-2+\frac{N}{\displaystyle N-2+\ddots }}} = [0;\overline{N-2}]_N.
\]
Hence if $\alpha = \xi_N-1$, we see that $T_{\alpha}(\alpha +1)=N/\xi_N - (N-1) = \alpha$, and therefore for all $0<\alpha \leq \xi_N-1$ all partial quotients must be at least $N-1$. 

So in case $\tfrac{t_0}{s_0}\in (1,\alpha+1 )$ we have that $d_1=N-1$, since $1<\tfrac{t_0}{s_0} < \xi_N$, and therefore that
\begin{equation}\label{eq:digit1}
   \frac{t_1}{s_1} = T_{N,\alpha}(\tfrac{t_0}{s_0}) = \frac{Ns_0-(N-1)t_0}{t_0}. 
\end{equation}
Also note that $\tfrac{t_1}{s_1} <1$, since for every $x\in (1,\alpha +1)$ we have that $T_{N,\alpha}(x)<T_{N,\alpha}(1)=1$. Once more applying $T_{\alpha}$ yields that:
$$
\frac{t_2}{s_2} = T_{N,\alpha}^2(\tfrac{t_0}{s_0}) = \frac{Nt_0-d_2(Ns_0-(N-1)t_0)}{Ns_0-(N-1)t_0},
$$
and since ${\rm gcd}\{t_2,s_2\} =1$ and $t_0>s_0$, we find that
$$
s_2 \leq Ns_0 - (N-1)t_0 < s_0.
$$
Since $(s_k)_{k\geq 0}$ is a sequence of positive integers, and $s_k>s_{k+1}$ in case $\tfrac{t_k}{s_k}\in [\alpha ,1)$, and $s_k>s_{k+2}$ in case $\tfrac{t_k}{s_k}\in (1,\alpha +1 )$, there must exist a positive integer $n$ such that $s_n=1$. As $1\in [\alpha ,\alpha +1]$ we must have that $t_n=1$, and the expansion of $t_0/s_0$ is periodic from that point on, with period length 1 and partial quotient $N-1$.
\end{proof}

\begin{corollary}\label{cor:N2ratfinorbit}
For $N=2$ we have that for all $\alpha\in(0,\sqrt{N}-1]$ all rational numbers are all eventually periodic, with period-length 1, and where the period consists of only the digit 1.
\end{corollary}

This follows from the fact that $\xi_2=\sqrt{2}$.

\subsection{Quadratic irrationals}
Just like for the regular continued fraction expansion, quadratic irrationals $x_0$ will be mapped to quadratic irrationals under the map $T_{N,\alpha}$. All the points in the orbit of $x_0$ under $T_{N,\alpha}$ are therefore roots of a quadratic equation with integer coefficients.  We will deduce recurrence relations for these coefficients. Let $x_0\in [\alpha, \alpha +1)$ be a positive quadratic irrational that is a solution to 
\begin{equation}\label{eq:quadirrstart}
    A_0x_0^2+B_0x_0+C_0=0
\end{equation}
where $A_0(\neq 0) ,B_0, C_0\in\mathbb{Z}$, ${\rm gcd}\{A_0,B_0,C_0\} = 1$, and define $x_1=T_{N,\alpha}(x_0)=\frac{N}{x_0}-d_1$. Then $x_0=\frac{N}{x_1+d_1}$, and substituting $x_0$ in \eqref{eq:quadirrstart} gives
\begin{align}
     A_0\left( \frac{N}{x_1+d_1}\right)^2+B_0\left( \frac{N}{x_1+d_1}\right) + C_0&=0\label{eq:quadirrnext}\\
     C_0x_1^2+(NB_0+2d_1C_0)x_1+N^2A_0+NB_0d_1+C_0d_1^2&=0,\label{eq:quadirrnext2}
\end{align}
where \eqref{eq:quadirrnext2} is found from \eqref{eq:quadirrnext} by multiplying by $(x_1+d_1)^2$. When setting $x_n=T_{N,\alpha}^n(x_0)$ we find that $x_n$ is a solution to 
\begin{equation}\label{eq:quadirrn}
    A_nx_n^2+B_nx_n+C_n=0
\end{equation}
where for $A_n (\neq 0), B_n, C_n\in\mathbb{Z}$ the following recurrence relations hold
\begin{align}
    &A_{n+1}=C_n\label{eq:recA_n}\\
    &B_{n+1}=NB_n+2d_{n+1}C_n\label{eq:recB_n}\\
    &C_{n+1}= N^2A_n+NB_nd_{n+1}+C_nd_{n+1}^2\label{eq:recC_n}
\end{align}
For the determinant we calculate 
\begin{align*}
    B_{n+1}^2-4A_{n+1}C_{n+1}&=(NB_n+2d_{n+1}C_n)^2-4C_n(N^2A_n+NB_nd_{n+1}+C_nd_{n+1}^2)\\
    &=N^2(B_n^2-4A_nC_n).
\end{align*}
Using this iteratively we find
\begin{equation}\label{eq:det}
    B_n^2-4A_nC_n=N^{2n}(B_0^2-4A_0C_0).
\end{equation}

From~\eqref{eq:det} it follows that in case $N=1$ and $\alpha = 0$ (which is the case of the \emph{regular continued fraction} (RCF) \emph{expansion}) we have that $B_n^2-4A_nC_n$ is a positive constant for all $n\geq 0$. From this one easily finds Lagrange's result that a quadratic irrational number $x$ has an eventually periodic RCF expansion. The converse result by Euler is even easier to prove. For a proof, see e.g.~\cite{HW}, or~\cite{Z,BFK}, where a similar result was obtained for the \emph{backward continued fraction expansion}.

\section{When $N\geq 2$ and each digit $d\in\mathcal{D}_{N,\alpha}$ are relatively prime}\label{sec:nodiv}
In this section we study the cases when ${\rm gcd}\{ N,d\} = 1$ for all $d\in\mathcal{D}_{N,\alpha}$. Let  
\begin{equation}\label{def:K}
K = \left\{ (N,\alpha): N\in\mathbb{N}_{\geq 2}, \  \alpha\in(0,\sqrt{N}-1] \text{ such that } {\rm gcd}\{ N,d\}=1 \text{ for all } d\in\mathcal{D}_{N,\alpha} \right\} .
\end{equation}
Note that in particular we have for $N\geq 5$ prime and $\alpha>1$ that $(N,\alpha)\in K$; see Figure~\ref{fig:my_label}.
\begin{figure}
    \centering
    \includegraphics[width=\textwidth]{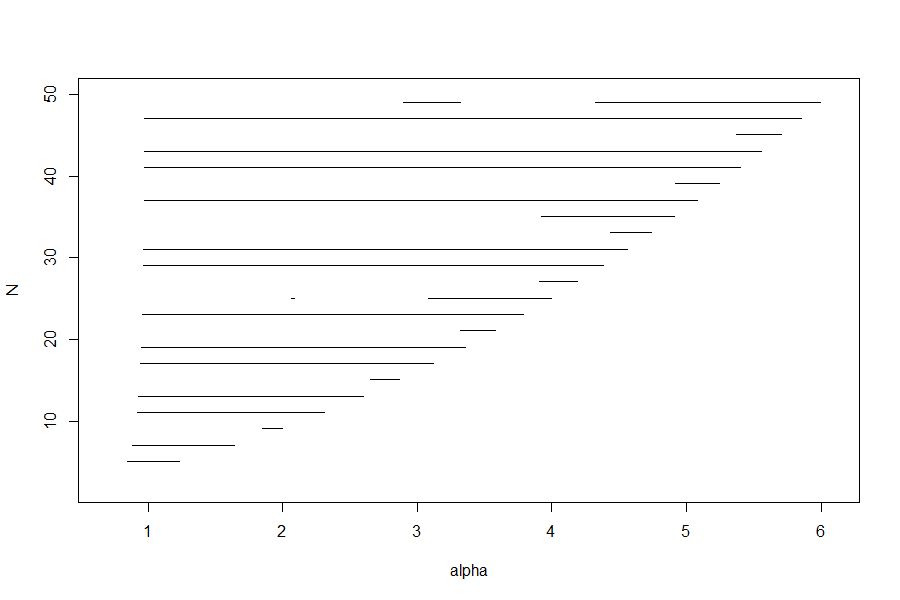}
    \caption{The set $K$, with $N$ on the $y$-axis and $\alpha$ on the $x$-axis.}
    \label{fig:my_label}
\end{figure}
Let us start with a positive observation.
\begin{observation}
For all $(N,\alpha)\in K$ and any $x\in[\alpha,\alpha+1]$ we have that $p_n(x)$ and $q_n(x)$ are co-prime. 
\end{observation}
It follows from \eqref{eq:noncoprimeprod} that the only possible common divisors of $p_n(x)$ and $q_n(x)$ are also divisors of $N$.  From \eqref{eq:recforqn} we find by an induction argument that $q_n$ does not have any common divisors with $N$ for $n\geq 0$. Here we use that each digit is relatively prime with $N$.\smallskip\

Since $0\not\in [\alpha,\alpha +1]$ we have that the $(N,\alpha)$-expansion of any $x\in [\alpha,\alpha +1]$ is infinite and unique. Note that if there exist $k,\ell \geq 0$, $k\neq \ell$, and $T_{N,\alpha}^k(x) = T_{N,\alpha}^{\ell}(x)$, then the $(N,\alpha )$-expansion of $x$ is (eventually) periodic. The converse of this also holds.

\begin{proposition}\label{prop:nonperrat}
Let $(N,\alpha)\in K$ with $\alpha>1$. Then, no rational number in $[\alpha,\alpha+1]$ is periodic.
\end{proposition}

\begin{proof}
Let $(N,\alpha)\in K$ with $\alpha>1$ and $x=t_0/s_0\in [\alpha ,\alpha +1]$
with $t_0,s_0\in\N$ be such, that ${\rm gcd}\{ t_0,s_0\}=1$. Now define recursively $t_{n+1}$ and $s_{n+1}$ for $n\geq 0$ by: 
\begin{equation}\label{eq:NumeratorsAndDenominators}
    t_{n+1} = Ns_n-d_{n+1}(x)t_n\quad \text{and}\,\, s_{n+1} = t_n.
\end{equation}
Clearly we have that $T_{N,\alpha}^n(\frac{t_0}{s_0})=\frac{t_n}{s_n}\in [\alpha ,\alpha +1)$, and since $\alpha>1$ yields that $t_{n+1}>s_{n+1}=t_n$, it follows that $(t_n)_{n\in\mathbb{N}}$ is a strictly increasing sequence. Note that here we do not demand that $t_n$ and $s_n$ are co-prime for $n\geq 1$. Now if $t_{n}$ and $s_{n}$ \emph{are} co-prime for every $n$ it immediately follows that the orbit of $x=t_0/s_0$ under $T_{N,\alpha}$ cannot be periodic. Left to show is when $t_n$ and $s_n=t_{n-1}$ are eventually not co-prime the $(N,\alpha$)-expansion of $x$ is still not periodic. To this end, we use the first statement in~\eqref{eq:NumeratorsAndDenominators}, reformulated as:
\begin{equation}\label{eq:NumeratorsAndDenominators2}
    Nt_{n-1} = t_{n+1}+d_{n+1}(x)t_n.
\end{equation}

First we let $(t_0,s_0)$ be such that $N$ is not a divisor of $t_0$ nor $s_0$. Then by~\eqref{eq:NumeratorsAndDenominators} and the definition of $K$ it immediately follows that $N$ is not a divisor of $t_1$ nor $s_1$ (in fact, if $N$ would be a divisor of $s_0$ this still would hold as $\text{gcd}\{ t_0,s_0\} = 1$ and $s_1=t_0$), and by induction we find that $N$ is not a divisor of $t_n$ nor $s_n$ for all $n\geq 1$. To see this, note that from~\eqref{eq:NumeratorsAndDenominators} we have that:
\[
t_{n+1}\, (\text{mod}\, N) \equiv Nt_{n-1}-d_{n+1}(x)t_n\, (\text{mod}\, N) \equiv -d_{n+1}(x)t_n\, (\text{mod}\, N),
\]
which implies that if $t_n$ is not divisible by $N$ then $t_{n+1}$ neither since all digits $d$ in our digit set $\mathcal{D}_{N,\alpha}$ are by definition~\eqref{def:K} of $K$ co-prime with $N$.

We will now show that for all $n\geq 1$ there is no prime $p\nmid N$ that divides $t_n$ and $s_n$. To see this, suppose this is not true. Let $k\geq 0$ be the \textit{first} time that some prime $p\nmid N$ is a divisor of both $t_{k+1}$ and $s_{k+1} (=t_{k})$ (and define $t_{-1}=s_0$). Then there exist $t_{k+1}^\prime$ and $t_k^\prime$ such that $t_{k+1}=pt_{k+1}^\prime$ and $s_{k+1}=t_k=pt_k^\prime $.
Then~\eqref{eq:NumeratorsAndDenominators2} yields that 
\[
 Nt_{k-1} = t_{k+1}+d_{k+1}(x)t_k = p(t_{k+1}^\prime +d_{k+1} t_k^\prime).
\]
Since $p\nmid N$ it follows that $t_{k-1} (= s_k)$ must be divisible by $p$ which leads to a contradiction with the definition of $k$.\medskip\

Second, suppose $t_0$ has $N$ as a factor; i.e., we can write $t_0=Nt_0^\prime$.
Then, since $s_1=t_0$,
$$
T_{N,\alpha}\left( \tfrac{t_{0}}{s_0}\right) = \frac{Ns_0-d_1t_0}{t_0} = \frac{Ns_0-d_1Nt_0^\prime}{Nt_0^\prime} = \frac{s_0-d_1t_0^\prime}{t_0^\prime}.
$$
Clearly for $k\geq 1$ each $t_k$ (and also $s_k$) has $N$ as a factor. In order to keep notation ``lean'' we redefine $t_1$ and $s_1$ from~\eqref{eq:NumeratorsAndDenominators} as $t_1 := s_0-d_1t_0^\prime$ and $s_1 := t_0^\prime$, which yields that $t_1 = s_0-d_1t_0^\prime<s_0<t_0$ where the second inequality follows from $\frac{t_0}{s_0}>1$. Let $t_2$ and $s_2$ be defined as in~\eqref{eq:NumeratorsAndDenominators}. If $t_1$ has $N$ as a factor we find that (after redefining $t_2$ and $s_2$ in the same way as we just redefined $t_1$ and $s_1$) that $t_2<t_1$. Repeating this process yields that $(t_n)_{n\in\mathbb{N}}$ is a sequence in $\mathbb{N}$ for which there must be an $k$ such that (after suitable re-definitions) $t_k$ does not have $N$ as a factor and we are back in the first case again. Since we defined in~\eqref{eq:NumeratorsAndDenominators} $s_{n+1} = t_n$ for all $n\geq 0$ we do not need to consider the case that $s_0$ has $N$ as a factor. As $\text{gcd}\{ t_0,s_0\} =1$ we cannot have that $N$ is a factor of both $t_0$ and $s_0$. So if $N|s_0$ we cannot have that $N|t_0$, and then we see that $N$ is not a divisor of $s_1$, as $s_1=t_0$. But then~\eqref{eq:NumeratorsAndDenominators} and the definition of $K$ yield that $N$ does not divide $t_1$.\smallskip\

Still it is possible that $N$ has a non-trivial factor $M$ which divides $t_0$ (or $s_0$): $M = {\rm gcd}\{ N,t_0\}$, with $M\not\in \{ 1,N\}$. Since $x = \tfrac{t_0}{s_0}$, where ${\rm gcd}\{t_0,s_0\}=1$, we also have that $x = \tfrac{Rt_0}{Rs_0}$, where $R=N/M$. So let us redefine $t_0$ and $s_0$ in such a way that $N$ is a factor of $t_0$, that $R$ is a factor of $s_0$, of course that $x=\tfrac{t_0}{s_0}$, and that ${\rm gcd}\{ t_0/R,s_0/R\} = 1$, where $R$ is a non-trivial factor of $N$. But then the above immediately yields (as we did not use there that ${\rm gcd}\{ t_0,s_0\} = 1$) that (after suitable re-definitions) there must be a positive integer $k$ for which $t_k$ does not have $N$ as a factor and we are back in the first case again.\medskip

\noindent We conclude that for $\alpha>1$ and $(N,\alpha)\in K$ no rational number in $[\alpha,\alpha+1]$ has a periodic expansion.
\end{proof}

The assumption $\alpha>1$ cannot be dropped. If $\alpha\leq 1$ then $1\in [\alpha,\alpha+1]$ and $1$ has a purely periodic expansion with period 1, having as expansion $1=[0;\overline{N-1}]_N$. Of course, all pre-images of $1$ are also (eventually) periodic.

\begin{remark}
{\rm In Proposition \ref{prop:FiniteExpansion} we have seen that for small $\alpha$, rationals are eventually periodic with period-length 1, and where the period consists of only the digit 1. There are more rational numbers that are fixed points of the map $T_{N,\alpha}$ for some $N,\alpha$. For example $2=[0;\overline{1}]_6=[0;\overline{2}]_8$ (or in general for $d\in\N$: $2=[0;\overline{d}]_{2d+4}$). These are points that are solutions to the equation $\frac{N}{x}-d=x$ so of the form $\frac{-d+\sqrt{d^2+4N}}{2}$. A natural question is: for a fixed pair $(N,\alpha)$, are rationals either eventually periodic with period 1 or non-periodic? This is not the case as other periods can be observed. For example: for $N=3$ and $\alpha=0.73$ we have that $\frac{40}{33}$ is eventually periodic with a pre-period of length 63 and a period length of 38. This is found by computing the orbit exactly with a script in {\tt R}. It is also possible to find smaller periods. For $N=9$ we have $2=[0;\overline{3,4}]_9$ and $\frac{3}{2}=[0;\overline{4,3}]_9$ which are the expansions of $2$ resp.\ $\tfrac{3}{2}$ when $\alpha=1.49$ for example.}
\end{remark}

\begin{proposition}\label{prop:nomatchingrat}
Let $(N,\alpha)\in K$ such that $\alpha=t_0/s_0\in\mathbb{Q}$ with ${\rm gcd}\{ t_0,s_0\} = 1$ and $t_0$ and $t_0+s_0$ are not divisible by $N$. Then there is no matching for $\alpha$.
\end{proposition}

\begin{proof}
Define $x_n= \frac{t_n}{s_n}=T_{N,\alpha}^n(\frac{t_0}{s_0})$ and $x_n^+=\frac{t_n^+}{s_n^+}=T_{N,\alpha}^n(\frac{t_0}{s_0}+1)$. Suppose that there is matching. Then there are minimal $i,j$ such that $\frac{t_{i+1}}{s_{i+1}}=\frac{t_{j+1}^+}{s_{j+1}^+} =: \frac{p}{q}$, with ${\text{gcd}}\{ p,q\} = 1$.
We find $x_i=\frac{N q}{d_{i+1}(\alpha)q+p}$ and $x_j^+=\frac{N q}{d_{j+1}(\alpha+1)q+p}$. From the proof of Proposition~\ref{prop:nonperrat} and our assumption on $t_0$ and $t_0+s_0$  we find that $t_i$ and $t_j^+$ are not divisible by $N$. Therefore we must have that 
\begin{align}
    d_{i+1}(\alpha)q+p &\equiv 0\,\, (\text{mod $N$})\label{eq:0modNfirst}\\
     d_{j+1}(\alpha+1)q+p &\equiv 0\,\, (\text{mod $N$}).\label{eq:0ModNsecond} 
\end{align}
Now due to the definition of $K$ neither $d_{i+1}(\alpha)$ nor $d_{j+1}(\alpha+1)$ has a divisor in common with $N$; in fact both $d_{i+1}(\alpha)$ and $d_{j+1}(\alpha+1)$ are smaller than $N$.
By minimality of $i$ and $j$ we must have that $d_{i+1}(\alpha) \neq d_{j+1}(\alpha+1)$, so without loss of generality we may assume that $d_{i+1}(\alpha)>d_{j+1}(\alpha+1)$. Then $d = d_{i+1}(\alpha) - d_{j+1}(\alpha+1)\in\{ 1,2,\dots, N-2\}$ and by subtracting~\eqref{eq:0ModNsecond} from~\eqref{eq:0modNfirst} we find $dq\equiv 0 \mod N$, which contradicts with our assumptions since $N$ is relative prime with $d$ and is not a divisor of $q$.
\end{proof}

\begin{corollary}
For $N\geq 5$ and odd, there are intervals of $(0,\sqrt{N}-1]$ that do not contain any matching interval.
\end{corollary}
\begin{proof}
For $N=5$ and $N=7$ we have that if $\alpha\in(1,\sqrt{N}-1]$ then $(N,\alpha)\in K$. The rationals that satisfy the conditions in Proposition \ref{prop:nomatchingrat} are dense in this interval. Now let $N\geq 9$ and odd and $\xi$ be such that $0 < \xi =\frac{N}{3+\xi}$. Then for all $\alpha\in(\xi,\sqrt{N}-1]$ we have that $\mathcal{D}_{N,\alpha}=\{1,2\}$. On this interval we also have that the rationals that satisfy the conditions in Proposition \ref{prop:nomatchingrat} are dense in this interval and therefore we cannot find a matching interval in it.
\end{proof}

\begin{proposition}
For $(N,\alpha)\in K$ and $N$ odd, for any quadratic irrational $x_0\in[\alpha,\alpha+1]$ such that ${\rm gcd}\{ C_0,N\} = 1$, where $C_0$ is from~\eqref{eq:quadirrstart}, we have that $x_0$ is not (eventualy) periodic.
\end{proposition}
\begin{proof}
Using the recurrence relations~\eqref{eq:recA_n}, \eqref{eq:recB_n} and~\eqref{eq:recC_n} for $A_n,B_n$ and $C_n$ we prove that they are co-prime. First we prove that $p$ does not divide $A_n,B_n$ and $C_n$ when $p$ prime is not equal to $N$ or a divisor of it. Suppose the contrary. Then there is an $n$ such that $A_{n+1},B_{n+1}$ and $C_{n+1}$ have $p$ as a divisor for the first time. Write
$A_{n+1}=p\widehat{A}_{n+1},B_{n+1}=p\widehat{B}_{n+1}$ and $C_{n+1}=p\widehat{C}_{n+1}$. Then using~\eqref{eq:recA_n}, \eqref{eq:recB_n} and~\eqref{eq:recC_n} we find
\begin{align}
    A_{n+1}=&\,\, p\widehat{A}_{n+1}=C_n & \text{gives } \quad C_n\equiv 0\,\, (\text{mod $p$}),\\
    B_{n+1}=&\,\, p\widehat{B}_{n+1}=NB_n+2d_nC_n & \text{gives } \quad  B_n\equiv 0\,\, (\text{mod $p$}),\\
    C_{n+1}=&\,\, p\widehat{C}_{n+1}=N^2A_n+NB_nd_n+C_nd_n^2 & \text{gives } \quad  A_n\equiv 0\,\, (\text{mod $p$}).
\end{align}
This is in contradiction with $n+1$ being the \emph{first} time that $A_{n+1},B_{n+1}$ and $C_{n+1}$ have $p$ as a divisor. Now let $p$ be equal to $N$ or a divisor of $N$. When $C_n\not \equiv 0\,\, (\text{mod $p$})$ we find by using~\eqref{eq:recA_n} that 
\[
A_{n+1}\not \equiv 0\,\, (\text{mod $p$}).
\]
By using \eqref{eq:recB_n} we find
\[
B_{n+1}\,\, (\text{mod $p$}) \equiv NB_n+2d_nC_n \not\equiv 0\,\, (\text{mod $p$}).
\]
Here we used that $N$ is odd and relatively prime with $d_n$ by definition of $K$. From~\eqref{eq:recC_n} we find
\[
C_{n+1}\,\, (\text{mod $p$}) \equiv N^2A_n+NB_nd_n+C_nd_n^2 \not\equiv 0\,\, (\text{mod $p$}) ,
\]
again by using that $d_n$ and $N$ are relatively prime.
Now by induction and our assumptions it follows that $p$ does not divide $A_n,B_n$ or $C_n$ for any $n\geq 1$. Now by using~\eqref{eq:det} we find that the sequences $(A_n)_{n\in\mathbb{N}},(B_n)_{n\in\mathbb{N}}$ or $(C_n)_{n\in\mathbb{N}}$ cannot be all bounded. Therefore the pigeon hole principle that is used in the proof for the regular continued fraction fails (cf.~\cite{HW}), and, in fact, we find that $x_0$ does not have a periodic orbit.
\end{proof}

\section{The case of $N=2$.}\label{sec:N2}
We will now take a closer look at the case of $N=2$. In this case, we know from Corollary \ref{cor:N2ratfinorbit} that  when we take $\alpha$ rational the orbits of $\alpha$ and $\alpha+1$ match. We will investigate when this matching is stable (find a matching interval that contains $\alpha$ on which we have the same matching exponents). We return to the study of M\"obius transformations.
Let $M_{\alpha,x,k}$ be the M\"obius transformation belonging to $x$ when iterated $k$ times over $T_{2,\alpha}$; cf.~(\ref{eq:Mxalphan}). Since the entries of the matrix are not necessarily co-prime we write $M\sim D$ for two matrices when they represent the same M\"obius transformation.
Now let
\[
R=
\left(\begin{matrix}
1 & 1 \\
0 & 1
\end{matrix}\right).
\]
We have the following proposition:
\begin{proposition}\label{prop:matchrel}
Let $x\in(0,\sqrt{2}-1]\cap \mathbb{Q}$. Then $x$ is contained in a matching interval with exponents $K,L$ if and only if
$RM_{x,x,K}\sim M_{x,x+1,L}$; see \eqref{eq:Mxalphan}.
\end{proposition}
\begin{proof}
($\Rightarrow$) Suppose $x$ is contained in a matching interval with exponents $K,L$. Then we have $RM_{x,x,K}(T_{2,x}^K(x))= M_{x,x+1,L}(T_{2,x}^L(x+1))$; cf.~(\ref{eq:xasMnTn}). For $x^\prime$ sufficiently close  to $x$, we have  $M_{x^\prime,x^\prime,K}=M_{x,x,K}$ and $M_{x^\prime,x^\prime+1,L}=M_{x,x+1,L}$. This holds for the following reason. Suppose there is no such neighborhood. Then the orbit of $x$ or $x+1$ has hit a discontinuity point before or at matching. If the orbit of $x+1$ hits a discontinuity point we find that matching happens with exponents $(0,L)$. In this case $RM_{x,x,K}=R$ which can never represent the same M\"obius transformation as $M_{x,x+1,L}$ since it has a zero entry. If the orbit of $x$ hits a discontinuity point we find that $x$ is purely periodic. But we know that the orbit of $x$ ends up in 1 and that and that $x\neq 1$. We find a contradiction. This is basically the same argument as in~\cite{L} on page 49, but there the contradiction is reached faster since for the 'classical families' rationals never have a periodic expansion. Note that we cannot carry over the argument for $N>2$. Now that we have a neighborhood such that  for all points in the neighborhood   $M_{x^\prime,x^\prime,K}=M_{x,x,K}$ and $M_{x^\prime,x^\prime+1,L}=M_{x,x+1,L}$
and because we are on a matching interval we have $RM_{x^\prime,x^\prime,K}(T_{2,x^\prime}^K(x^\prime))= M_{x^\prime,x^\prime+1,L}(T_{2,x^\prime}^L(x^\prime+1))$. We find that $RM_{x,x,K}(T_{2,x^\prime}^K(x^\prime))= M_{x,x+1,L}(T_{N,x^\prime}^L(x^\prime+1)) $ for all $x^\prime$ sufficiently close, which is only possible when $RM_{x,x,K}\sim M_{x,x+1,L}$.\\

($\Leftarrow$) Now suppose that $RM_{x,x,K}\sim M_{x,x+1,L}$ then from \eqref{eq:xasMnTn}, on the one hand we have
\[
x+1= RM_{x,x,K}(T_{2,x}^K(x))
\]
and 
\[
x+1=M_{x,x+1,L}(T_{2,x}^L(x+1)).
\]
Using this, and $RM_{x,x,K}\sim M_{x,x+1,L}$ we find $T_{2,x}^K(x)=T_{2,x}^L(x+1)$. Now let $x^\prime$ be close enough to $x$ such that $M_{x^\prime,x^\prime,K}=M_{x,x,K}$ and $M_{x^\prime,x^\prime+1,L}=M_{x,x+1,L}$. Then we also find $T_{2,x^\prime}^K(x^\prime)=T_{2,x^\prime}^L(x^\prime+1)$ and therefore, $x$ is contained in a matching interval with exponents $K,L$.
\end{proof}

We now show that there are infinitely many rationals that are not in any matching interval. Note that for Tanaka-Ito continued fractions such rationals also exist but that matching holds for almost every parameter in that case; see~\cite{CLS}. In the case of Nakada's $\alpha$-continued fractions (\cite{N}) and Katok and Ugarcovici's continued fractions (\cite{KU}) they do not exist, giving that all rationals are contained in some matching interval. We call rationals that are not in any matching interval \emph{bad rationals}. First we show that we have infinitely many of them.

\begin{proposition}\label{prop:badrat}
For $N=2$ and $\alpha_n=\frac{1}{2^n}$, $n\geq 3$, we have that $\alpha_n$ is a bad rational.
\end{proposition}

\begin{proof}
We have that $\alpha_n=[0;2^{n+1}-1,\overline{1}]_{\alpha_n}$ and $\alpha_n+1=[0;1,2,2^{n-1}-1,3,\overline{1}]_{\alpha_n}$. But then we have that $T_{2,\alpha_n}(\alpha_n)=T_{2,\alpha_n}^4(\alpha_n+1)$,
\[
RM_{\alpha_n,\alpha_n,1}=
\left(\begin{matrix}
1 & 2^{n+1}+1 \\
1 & 2^{n+1}-1
\end{matrix}\right) ,
\quad \text{and } \quad
 M_{\alpha_n,\alpha_n+1,4}=
 \left(\begin{matrix}
2^{n+1} & 3*2^{n+1}+8 \\
2^{n+1}-2& 3*2^{n+1}+2
\end{matrix}\right).
\]
This gives
\[
RM_{\alpha_n,\alpha_n,1}
\left(\begin{matrix}
0 & 2 \\
1 & 1
\end{matrix}\right)^k \equiv
\left(\begin{matrix}
1 & 1 \\
1 & 1
\end{matrix}\right)\,\, (\text{mod $2$}).
\]
Let $\widehat{M}_{\alpha_n,\alpha_n+1,4}=\frac{1}{2}M_{\alpha_n,\alpha_n+1,4}$ then $\widehat{M}_{\alpha_n,\alpha_n+1,4}\sim M_{\alpha_n,\alpha_n+1,4} $ and 
\[
\widehat{M}_{\alpha_n,\alpha_n+1,4}
\left(\begin{matrix}
0 & 2 \\
1 & 1
\end{matrix}\right)^k\equiv
\left(\begin{matrix}
0 & 0 \\
1 & 1
\end{matrix}\right)\,\, (\text{mod $2$}).
\]
We find that for all $K$ and $L$ we have
$RM_{x,x,K}\not\sim M_{x,x+1,L}$. Therefore, $\alpha_n$ is bad.
\end{proof}
On the other hand we also find infinitely many matching intervals.
\begin{theorem}\label{theorem:matchingintervals}
Let $k\geq 0$.
On the following intervals we have matching
\begin{itemize}
\item[(i)]
\[
\left(\frac{-17-8k+\sqrt{369+304k+64k^2}}{10+4k},\frac{-2-k+\sqrt{6+5k+k^2}}{2+k}\right)
\] with matching exponents $(3,5)$ around $\alpha_k=\frac{2}{9+4k}$ that matches with exponents $(1,5)$.
\item[(ii)]
\[
\left(
\frac{-81 - 32 k + \sqrt{8289 + 5824 k + 1024 k^2}}{54 + 20k},
\frac{-10-4k+\sqrt{132 + 92 k + 16 k^2}}{8 + 3k}
\right)
\] with matching exponents $(5,5)$ around $\alpha_k=\frac{8}{43+16k}$ that matches with exponents $(2,4)$.
\item[(iii)]
\[
\left(
\frac{-133 - 52k + \sqrt{24033 + 16120k + 2704k^2}}{122 + 44k},
\frac{-273-104k+\sqrt{13}\sqrt{7061 + 4848 k + 832 k^2}}{166 + 60k}
\right)
\] with matching exponents $(6,6)$ around $\alpha_k=\frac{13}{72+26k}$ that matches with exponents $(4,6)$.
\item[(iv)]
\[
\left(
\frac{-363 - 120k + \sqrt{3}\sqrt{53603 + 32080k + 4800k^2}}{242 + 76k},
\frac{-45-15k+\sqrt{15}\sqrt{170 + 101 k + 15 k^2}}{35 + 11k}
\right)
\] with matching exponents $(7,7)$ around $\alpha_k=\frac{30}{191+60k}$ that matches with exponents $(4,6)$.

\end{itemize}
\end{theorem}

\begin{proof}
We will prove the intervals in ($i$) and give the characteristics of the others in Table \ref{tab:match}.
First we will show that $\alpha_k=[0;8+4k,\overline{1}]_{\alpha_k}$ and $\alpha_k+1=[0;1,2,k+1,2,2,\overline{1}]_{\alpha_k}$. To see these  are \textit{some} expansions, note that $1=[0;\overline{1}]$ and therefore we have
\[
\alpha_k=\frac{2}{\displaystyle 8+4k+\frac{2}{\displaystyle 1+\frac{2}{\displaystyle 1+\ddots}}}=\frac{2}{9+4k}
\]
and
\[
\alpha_k+1=\frac{2}{\displaystyle 1+\frac{2}{\displaystyle 2+\frac{2}{\displaystyle 1+k+\frac{2}{\displaystyle 2+\frac{2}{2+1}}}}}=\frac{4k+11}{4k+9}=\frac{2}{4k+9}+1.
\]
To see these are \textit{the} $(N,\alpha_k)$-expansions let $\tilde{\alpha_k}$ be the sequence $8+4k,\overline{1}$ and $\tilde{\alpha_k}+1$ the sequence $1,2,k+1,2,2,\overline{1}$. Then, by using the alternating ordering and the shift map $\sigma$, we see that
\[
\tilde{\alpha_k}\preceq \sigma^n(\tilde{\alpha_k})\prec \tilde{\alpha_{k}}+1.
\]
Therefore we found the expansions of $\alpha_k$ and $\alpha_k+1$.We could also have used the map $T_{2,\alpha_k}$, which immediately yields the result for $\alpha_k$, as
$$
\frac{2}{\phantom{Z}\frac{2}{9+4k}\phantom{Z}} - (8+4k) = 1,
$$
but which is more involved for $\alpha_k+1$. Note that $\alpha_k$ and $\alpha_k+1$ match with matching exponents $(1,5)$. Now for the matrices $M_{\alpha_k,\alpha_k,3}$ and $M_{\alpha_k,\alpha_k+1,5}$ we have
\[
RM_{\alpha_k,\alpha_k,3}=
\left(\begin{matrix}
4k+12 & 12k+32 \\
4k+10 & 12k+26
\end{matrix}\right),\quad
M_{\alpha_k,\alpha_k+1,5}=
\left(\begin{matrix}
8k+24 & 24k+64 \\
8k+20 & 24k+52
\end{matrix}\right).
\]
We find that $M_{\alpha_k,\alpha_k,3}\sim M_{\alpha_k,\alpha_k+1,5}$. Using Proposition \ref{prop:matchrel}, we now found a matching interval. To find the boundaries, we need to identify the largest interval for which all $\alpha$ in the interval start with the same 3 digits as $\alpha_k$ in their $(N,\alpha)$-expansion and $\alpha+1$ then starts with the same 5 digits as $\alpha_k+1$.
Let us define
\[
\Delta^\alpha(d_1,d_2,\ldots, d_n)=\{\alpha\in (0,\sqrt{N}-1):\alpha=[0;d_1,d_2,\ldots,d_n,\ldots]_{\alpha}\}.
\]
The boundaries  of $\Delta^\alpha(d_1,d_2,\ldots, d_n)$ are given by the equations:
\[
\alpha_1=[0;d_1,d_2,\ldots,d_{n}+1,\alpha_1]_{\alpha_1} = [0;\overline{d_1,d_2,\ldots,d_{n}+1}]_{\alpha_1},
\]
and
\[
\alpha_2=\begin{cases}
    [0;d_1,d_2,\ldots, d_n,\alpha_2]_{ \alpha_2}, & \text{when } d_n>1,\\
    [0;d_1,d_2,\ldots, d_{n-1},\alpha_2+1]_{ \alpha_2}, & \text{when } d_n=1,
\end{cases}
\]
see Figure \ref{fig:cyl} for an explanation. 
	\begin{figure}[ht]
		\centering
		\subfigure{\begin{tikzpicture}[scale=5]
				\draw(0,0)node[below]{\small $\alpha$}--(1,0)node[below]{\small $\alpha+1$}--(1,1)--(0,1)node[left]{\small $\alpha+1$}--(0,0);

				\draw[thick, purple!50!black, smooth, samples =20, domain=2/3:1] plot(\x-0.25,{2/\x-2});
				\draw(2/3-0.25,0)circle(0.35pt)[color=black!,fill=white!, fill opacity=1];
    			\draw(0.75,0)circle(0.35pt)[color=black!,fill=white!, fill opacity=1];
                \draw(0.6,0)circle(0.35pt)[color=black!,fill=white!, fill opacity=1]node[above]{\tiny $T_\alpha^{n-1}(\alpha)$};
				
				\draw[dotted](2/3-0.25,0)node[below]{\small $\frac{N}{d+1+\alpha}$}--(2/3-0.25,1);
    \draw[dotted](0.75,0)node[below]{\small $\frac{N}{d+\alpha}$}--(0.75,1);
				
\end{tikzpicture}}
  \caption{When picking $\alpha^\prime$ extremely close to $\alpha$ we can ensure to pick the same branches $\frac{N}{x}-d$ as for $\alpha$. If we pick the first $n-1$ digits the same, the $(N,\alpha^\prime)$-expansion of $\alpha^\prime$ will start with $d_1,\ldots, d_{n-1}$. To get also the $n^{\text{th}}$ digit the same we need to have that $\frac{N}{d+1+\alpha^\prime}<T_{\alpha^\prime}^{n-1}(\alpha^\prime)\leq \frac{N}{d+\alpha^\prime}$  with $d$ the $n^{\text{th}}$ digit. This will give you the boundary cases $\alpha_1$ and $\alpha_2$.}\label{fig:cyl}
  \end{figure}
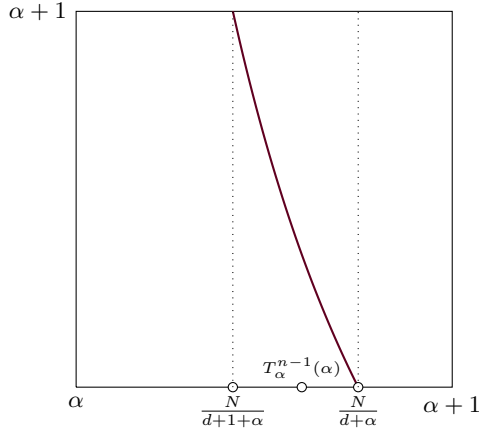

When $n$ is odd the interval is given by $(\alpha_1,\alpha_2)$ and given by $(\alpha_2,\alpha_1)$ whenever $n$ is even\footnote{strictly speaking we should include $\alpha_2$ in the interval but this might lead to degenerate cases which we wish to avoid.}.
In other words, all $\alpha$ that start with the digits $d_1,d_2,\ldots, d_n$ in their $(2,\alpha)$-expansion. Likewise we define
\[
\Delta^{\alpha+1}(d_1,d_2,\ldots, d_n)=\{\alpha\in (0,\sqrt{N}-1):\alpha+1=[0;d_1,d_2,\ldots,d_n,\ldots]_{\alpha}\}.
\]
Similarly as for $\alpha$, the boundaries are given by the equations
\[
\alpha_1+1=[0;d_1,d_2,\ldots,d_{n}+1,\alpha_1]_{ \alpha_1},
\]
and
\[
\alpha_2+1=\begin{cases}
    [0;d_1,d_2,\ldots, d_n,\alpha_2]_{ \alpha_2}, & \text{when } d_n>1,\\
    [0;d_1,d_2,\ldots, d_{n-1},\alpha_2+1]_{ \alpha_2}, & \text{when } d_n=1,
\end{cases}
\]
Then the matching interval is given by $\Delta^\alpha(8+4k,1,1)\cap \Delta^{\alpha+1}(1,2,k+1,2,2)$. Calculating the boundaries we find
\[
\Delta^\alpha(8+4k,1,1)=\left(\frac{-17-8k+\sqrt{369+304k+64k^2}}{10+4k},\frac{-2-k+\sqrt{6+5k+k^2}}{2+k}\right)
\]
and
\[
\Delta^{\alpha+1}(1,2,k+1,2,2)=\left(\frac{-17-8k+\sqrt{369+304k+64k^2}}{10+4k},\frac{-3k-6+\sqrt{9k^2+42k+51}}{2k+5}\right).
\]
Since $\Delta^\alpha(8+4k,1,1)\subset\Delta^{\alpha+1}(1,2,k+1,2,2)$, the matching interval is given by
\[
\left(\frac{-17-8k+\sqrt{369+304k+64k^2}}{10+4k},\frac{-2-k+\sqrt{6+5k+k^2}}{2+k}\right).
\]
The proof of the other intervals goes exactly the same. Therefore we omit it and instead give the characteristic data in Table \ref{tab:match}.
\begin{table}[]
    \centering
    \begin{tabular}{|c|l|l|c|}\hline
 $\alpha_k$     & expansion of $\alpha_k$ & expansion of $\alpha_k+1$ & $RM_{\alpha_k,\alpha_k,K}$  \\\hline
       $\frac{2}{9+4k}$ & $[0;8+4k,\overline{1}]$ & $[0;1,2,k+1,2,2,\overline{1}]$ &
$\left(\begin{matrix}
4k+12 & 12k+32 \\
4k+10 & 12k+26
\end{matrix}\right)$ \\ \hline
$\frac{8}{43+16k}$   & $[0;10+4k,2,2,\overline{1}]$ & $[0;1,2,k+2,10,2,\overline{1}]$ &
$\left(\begin{matrix}
40k+128 & 88k+280 \\
40k+108 & 88k+236
\end{matrix}\right)$
\\ \hline 
$\frac{13}{72+26k}$  & $[0;10+4k,1,2,5,\overline{1}]$ & $[0;1,2,k+2,7,4,2,\overline{1}]$ &
$\left(\begin{matrix}
120k+392 & 296k+968 \\
120k+332 & 296k+820
\end{matrix}\right)$\\ \hline 
$\frac{30}{191+60k}$ & $[0;12+4k,2,2,2,2,\overline{1}]$ & $[0;1,2,k+2,2,2,12,2,\overline{1}]$ &
$\left(\begin{matrix}
304k+1120 & 656k+2416 \\
304k+968 & 656k+2088
\end{matrix}\right)$\\ \hline 
    \end{tabular}\vspace{1em}
 \caption{The $(2,\alpha_k)$-expansions of $\alpha_k$ and $\alpha_k+1$, from Theorem \ref{theorem:matchingintervals} and the corresponding matrix for $\alpha_k$. Note that the corresponding matrix for $\alpha_k+1$ is the same matrix except for the case of $\alpha_k=\frac{2}{9+4k}$ where the entries of the corresponding matrix are twice as much.}
    \label{tab:match}
\end{table}   
\end{proof}

\begin{remark}
{\rm As a final remark, we would like to mention that the matching intervals in Theorem \ref{theorem:matchingintervals} are hard to find since they are extremely small. Moreover, the rationals in Proposition \ref{prop:badrat} are not the only bad ones. With computer simulations one can find many more. It would be interesting to find out whether there is an interval that does not contain any matching interval in the case of $N=2$.}
\end{remark}

\bibliographystyle{alpha}
\bibliography{MatchingNexpansions}


\end{document}